\documentclass[a4paper,12pt]{article}
\usepackage{latexsym}
\usepackage{amsmath}
\usepackage{amssymb}
\usepackage{amscd}
\usepackage{graphicx}
\usepackage[usenames,dvipsnames]{color}

\newtheorem{theo}{Theorem}[section]
\newtheorem{lemma}[theo]{Lemma}
\newtheorem{coro}[theo]{Corollary}
\newtheorem{proposition}[theo]{Proposition}
\newtheorem{definition}[theo]{Definition}
\newtheorem{conjecture}[theo]{Conjecture}
\newcommand{\EQ}{\begin{equation}}
\newcommand{\EN}{\end{equation}}
\newenvironment{proof}{\indent{\em Proof.  }}{\hfill{$\Box$\bigskip\newline}}

\newcommand{\Aut}{\mbox{\rm Aut}}

\newcommand{\Supp}{\operatorname{Supp}}
\newcommand{\wt}{\mbox{\rm wt}}

\newcommand{\ba}{{\bf a}}

\newcommand{\bc}{{\bf c}}

\newcommand{\be}{{\bf e}}

\newcommand{\rr}{\bar{r}}

\newcommand{\by}{{\bf y}}
\newcommand{\bx}{{\bf x}}
\newcommand{\bz}{{\bf z}}
\newcommand{\bv}{{\bf v}}
\newcommand{\bu}{{\bf u}}
\newcommand{\bw}{{\bf w}}

\newcommand{\bo}{{\bf 0}}

\newcommand{\Det}{\operatorname{det}}

\newcommand{\F}{\mathbb{F}}

\title{Families of nested completely regular
codes and distance-regular graphs\footnote{This work has been
partially supported by the Spanish MICINN grant TIN2013-40524-P;
the Catalan grant 2009SGR1224 and also by the
Russian fund of fundamental researches 12-01-00905.\newline \indent $^1$J. Rif\`{a} and J. Borges
are with the Department of Information and Communications
Engineering, Universitat Aut\`{o}noma de Barcelona.\newline \indent $^2$ V. Zinoviev is with the
A. A. Harkevich Institute for Problems of Information
Transmission, Russian Academy of Sciences.}}

\author{J. Borges$^1$, J. Rif\`{a}$^1$, V. A. Zinoviev$^2$}
\begin{document}

\maketitle

\begin{abstract}
In this paper infinite families of linear binary nested completely
regular codes are constructed. They have covering radius $\rho$
equal to $3$ or $4$, and are $1/2^i$-th parts, for
$i\in\{1,\ldots,u\}$ of binary (respectively, extended binary)
Hamming  codes of length $n=2^m-1$ (respectively, $2^m$), where
$m=2u$. In the usual way, i.e., as coset graphs, infinite families
of embedded distance-regular coset graphs of diameter $D$ equal
to $3$ or $4$ are constructed. In some cases, the constructed codes are
also completely transitive codes and the corresponding coset
graphs are distance-transitive.
\end{abstract}

\section{Introduction}\label{int}

Let $\F_q$ denote the finite field with $q\geq 2$ elements, $q$ being
a prime power. For a vector
$\bx \in \F_q^n$ denote by $\wt(\bx)$ its Hamming weight (i.e., the number
of its nonzero positions).
For every two vectors $\bx=(x_1,\ldots, x_n)$ and $\by = (y_1, \ldots, y_n)$
from $\F_q^n$ denote by $d(\bx, \by)$ the Hamming distance between $\bx$
and $\by$ (i.e., the number of positions $i$, where $x_i \neq y_i$).
We use the standard notation $[n,k,d]$ for a binary linear code $C$
of length $n$, dimension $k$ and minimum distance $d$ over the
binary field $\F_2$.

The automorphism
group $\Aut(C)$ of $C$ consists of all $n\times n$ binary permutation
matrices $M$, such that $\bc M \in C$ for all $\bc \in C$.  Note
that the automorphism group $\Aut(C)$ coincides with the subgroup
of the symmetric group $S_n$ consisting
of all $n!$ permutations of the $n$ coordinate positions which send
$C$ into itself. $\Aut(C)$ acts in a natural way over the set of cosets
of $C$: $\pi(C+\bv)=C+\pi(\bv)$ for every $\bv\in\F_2^n$ and $\pi\in\Aut(C)$.

For any $\bv \in \F_2^n$ its {\em distance} to the code
$C$ is $d(\bv,C)=\min_{\bx \in C}\{ d(\bv, \bx)\}$ and the {\em covering
radius} of the code $C$ is $\rho=\max_{\bv \in \F_2^n} \{d(\bv, C)\}$.
Let $J = \{1,2, \ldots, n\}$ be the set of coordinate positions of
vectors from $\F_2^n$. Denote by $\Supp(\bx)$ the {\em support} of
the vector $\bx = (x_1, \ldots, x_n)\in \F_2^n$, i.e.,
$\Supp(\bx)~=~\{j \in J:~x_j \neq 0\}$. Say that two vectors
$\bx,\by \in\F_2^n$ are {\em neighbors} if $d(\bx,\by)=1$ and also say that vector $\bx$ {\em covers}
vector $\by$ if $\Supp(\by)\subseteq \Supp(\bx)$.

For a given binary code $C$ with the zero codeword and with covering radius
$\rho=\rho(C)$ define
$$
C(i)~=~\{\bx \in \F_2^n:\;d(\bx,C)=i\},\;\;i=1,2,\ldots,\rho,
$$
and
$$
C_i~=~\{\bc \in C:\;\wt(\bc) = i\},\;\;i=0,1,\ldots, n.
$$

%Denote by $W(\bx)$ the set of all neighbors of $\bx$.

\begin{definition}\label{de:1.1}  A code $C$ with covering radius
$\rho=\rho(C)$ is completely regular,
if for all $l\geq 0$ every vector $x \in C(l)$ has the same
number $c_l$ of neighbors in $C(l-1)$ and the same number
$b_l$ of neighbors in $C(l+1)$. Also, define
$a_l = (q-1)n-b_l-c_l$ and note that $c_0=b_\rho=0$.\\
Alternatively, $C$ is completely regular if and only if the weight
distribution of any coset $C+\bv$ of weight $i$, for
$i\in\{0,\ldots,\rho\}$ is uniquely determined by the minimum
weight $i$ of $C+\bv$.\\ For a completely regular code, define
$(b_0, \ldots, b_{\rho-1}; c_1,\ldots, c_{\rho})$ as the intersection
array of $C$.
\end{definition}

\begin{definition}\cite{sole} A binary linear code $C$ with
covering radius $\rho$ is completely transitive if $\Aut(C)$
has $\rho + 1$ orbits when acts on the cosets of $C$.
\end{definition}

Since two cosets in the same orbit have the same weight
distribution, it is clear that any completely transitive code is
completely regular.

Existence and enumeration of completely regular and completely
transitive codes are open hard problems (see
\cite{bro2,del1,neu,sole,Giu} and references there). The purpose of this
paper is to construct nested infinite families of completely
regular codes with covering radius $\rho=3$ and $\rho=4$. When $m$
is growing the length of the chain of these nested codes (with
constant covering radius) is also growing. For length $n=2^m-1$,
where $m=2u$, each family is formed by $u$ nested completely
regular codes of length $n$ with the same covering radius
$\rho=3$. The last code in the nested family, so the code with the
smallest cardinality is a $1/2^u$-th part of a Hamming code of
length $n$. These last codes are known to be completely regular
codes due to Calderbank and Goethals \cite{cal,kas}. These nested
families of completely regular codes and their extended codes
induces infinite families of embedded distance-regular coset
graphs with diameters $3$ and $4$, which also give interesting
families of embedded covering graphs. We point out that in some
cases such completely regular codes are also completely transitive
and hence the corresponding coset graphs are also distance
transitive.

\section{Preliminary results}\label{pre}

\begin{definition}\label{de:1.2}
Let $C$ be a binary code of  length $n$ and let $\rho$ be its
covering radius. We say that $C$ is {\em uniformly packed} in the
wide sense, i.e., in the sense of \cite{bas1}, if there exist
rational numbers $~\beta_0,\ldots,\beta_{\rho}~$ such that for any
$\bv\in\F_2^n$
\EQ
\sum_{k=0}^{\rho}\beta_k\,\alpha_k(\bv)~=~1,
\EN
where $\alpha_k(\bv)$ is the number of codewords at distance $k$ from
$\bv$.
\end{definition}

Let $C$ be a linear code. Denote by $s$ the number of nonzero weights in its dual code $C^\perp$.
Following \cite{del1}, we call $s$ the {\em external distance} of $C$.

\begin{lemma}\label{lem:2.1}
Let $C$ be a code with  covering radius
$\rho$ and external distance $s$. Then:
\begin{description}
\item[(i)]$~~\cite{del1}$ $\rho \leq s$.
\item[(ii)]$~\cite{bas2}$ $\rho = s$ if and only if $C$ is uniformly
packed in the wide sense.
\item [(iii)]$\cite{bro2}$ If $C$ is completely regular, then it is
uniformly packed in the wide sense.
\end{description}
\end{lemma}

\begin{lemma}\label{lem:2.2} $\cite{rif2}$
Let $C$ be a linear completely regular $[n,k,d]$ code with covering radius
$\rho$ and intersection array $(b_0, \ldots , b_{\rho-1}; c_1, \ldots c_{\rho})$.
Let $\mu_i$ denote the number of cosets of $C$ of weight $i$,
where $i=0,1,\ldots, \rho$. Then the following equality holds:
$$
b_i\mu_i~=~c_{i+1}\mu_{i+1},~~~i=0, \ldots, \rho-1.
$$
\end{lemma}

Next, following \cite{bro2}, we give some facts on distance-regular graphs.
Let $\Gamma$ be a finite connected simple graph (i.e., undirected, without loops and multiple
edges). Let $d(\gamma, \delta)$ be the distance between two vertices
$\gamma$ and $\delta$ (i.e., the number of edges in the minimal path between
$\gamma$ and $\delta$). The {\em diameter} $D$ of $\Gamma$ is its largest distance.
Two vertices $\gamma$ and $\delta$ from $\Gamma$ are
{\em neighbors} if $d(\gamma, \delta) = 1$. Denote
$$
\Gamma_i(\gamma) ~=~\{\delta \in \Gamma:~d(\gamma, \delta) = i\}.
$$

An {\em automorphism} of a graph $\Gamma$ is a
permutation $\pi$ of the vertex set of $\Gamma$ such that, for all
$\gamma, \delta \in \Gamma$ we have $d(\gamma,\delta)=1$, if and
only if $d(\pi\gamma,\pi\delta)=1$. Let $\Gamma_i$ be the graph with the same vertices of $\Gamma$, where an edge $(\gamma, \delta)$
is defined when the vertices $\gamma, \delta$ are at distance $i$
in $\Gamma$. Clearly, $\Gamma_1=\Gamma$. The graph $\Gamma$ is called
{\em primitive} if $\Gamma$ and all $\Gamma_i$ ~($i=2, \ldots, D$)
are connected. Otherwise, $\Gamma$ is called {\em imprimitive}. A graph is called \textit{complete} (or a \textit{clique}) if any two of its vertices are adjacents.

A connected graph $\Gamma$ with diameter $D\geq 3$ is called {\em antipodal}
%if the vertices at distance $D$ from a given vertex are all at
%distance $D$ from each other \cite{bro2}, which menas that
if the
graph $\Gamma_D$ is a disjoint union of cliques~\cite{bro2}. Such a graph is
imprimitive by definition. In this case, the {\em folded graph, or
antipodal quotient} of $\Gamma$ is defined as the graph
$\bar{\Gamma}$, whose vertices are the maximal cliques (which are
called {\em fibres}) of $\Gamma_D$, with two adjacent if and only
if there is an edge between them in $\Gamma$. If, in addition,
each edge $\gamma\in \Gamma$ has the same valency as its image under folding, then
$\Gamma$ is called an {\em antipodal covering graph} of
$\bar{\Gamma}$. If, moreover, all fibres of $\Gamma_D$ have the
same size $r$, then $\Gamma$ is also called an {\em antipodal
$r$-cover} of $\bar{\Gamma}$.

\begin{definition}\cite{bro2}\label{14:de:1.3}
A simple connected graph $\Gamma$ is called {\em distance-regular}
if it is regular of valency $k$, and if for any two vertices
$\gamma, \delta \in \Gamma$ at distance $i$ apart, there are precisely
$c_i$ neighbors of $\delta$ in $\Gamma_{i-1}(\gamma)$ and $b_i$
neighbors of $\delta$ in $\Gamma_{i+1}(\gamma)$.
Furthermore, this graph is called {\em distance-transitive}, if
for any pair of vertices $\gamma, \delta$ at distance
$d(\gamma, \delta)$ there is an automorphism $\pi$ from
$\mbox{\rm Aut}(\Gamma)$ which move this pair $(\gamma,\delta)$ to any other given
pair $\gamma', \delta'$  of vertices at the same distance
$d(\gamma, \delta) = d(\gamma', \delta')$.
\end{definition}

The sequence $(b_0, b_1, \ldots, b_{D-1};
c_1, c_2, \ldots, c_D)$, where $D$ is the diameter of $\Gamma$,
is called the {\em intersection array} of $\Gamma$. The numbers
$c_i, b_i$, and $a_i$, where $a_i=k- b_i - c_i$, are called
{\em intersection numbers}. Clearly $b_0 = k,~~b_D = c_0 = 0,~~c_1 = 1.$

Let $C$ be a linear completely regular code with covering radius $\rho$ and
intersection array  $(b_0, \ldots , b_{\rho-1}; c_1, \ldots c_{\rho})$.
Let $\{B\}$ be the set of cosets of $C$. Define the graph $\Gamma_C$, which is
called the {\em coset graph of $C$}, taking all different cosets $B = C+ {\bf x}$ as
vertices, with two vertices $\gamma = \gamma(B)$ and $\gamma' = \gamma(B')$
adjacent, if and only if the cosets $B$ and $B'$ contain neighbor vectors,
i.e., there are ${\bf v} \in B$ and ${\bf v}' \in B'$ such that $d({\bf v}, {\bf v}') = 1$.

\begin{lemma}\cite{bro2,ripu}\label{lem:2.5}
Let $C$ be a linear completely regular code with covering radius $\rho$ and
intersection array  $(b_0, \ldots , b_{\rho-1}; c_1, \ldots c_{\rho})$
and let $\Gamma_C$ be the coset graph of $C$. Then $\Gamma_C$ is
distance-regular of diameter $D=\rho$ with the same intersection array.
If $C$ is completely transitive, then $\Gamma_C$ is distance-transitive.
\end{lemma}

\begin{definition}\label{de:2.1}
A set $T$ of vectors $\bv \in \F_2^n$ of weight $w$ is
a $t$-{\em design}, denoted by $T(n, w, t, \lambda)$, if
for any vector $\bz \in \F_2^n$ of weight $t$, $1 \leq t \leq w$,
there are precisely $\lambda$ vectors $\bv_i,~i=1,\ldots,\lambda$
from  $T(n, w, t, \lambda)$, each of them covering ${\bf z}$.
\end{definition}

The following well known fact directly
follows from the definition of completely regular code.

\begin{lemma}\label{lem:2.6}
Let $C$ be a completely regular code with minimum distance $d$
and containing the zero codeword. Then the set $C_w$  (of codewords
of $C$ of weight $w$),~$d \leq w \leq n$ forms a $t$-design, if it is not
empty, where $t=e$, if $d=2e+1$ and $t=e+1$, if $d=2e+2$.
\end{lemma}

Let $H_m$ denote a binary matrix of size $m\times n$, where $n=2^m-1$, whose
columns are all different nonzero binary vectors of length $m$, i.e., $H_m$
is a parity check matrix of a binary (Hamming) $[n,n-m,3]$-code, denoted by
$\mathcal{H}_m$.
%We have the following description of completely regular
%codes \cite{bor3}, which are halves of $\mathcal{H}_m$.
%
%\begin{lemma}\label{lem:2.60}\cite{bor3}
%Let $C$ be a $[n,k,3]$ binary (linear) code with covering radius $\rho=3$,
%such that $C(\rho)=C+\bc$, for some $\bc\in C(\rho)$.\\
%(i)~~~~ The code $C$ is completely regular, if and only if $C_3$ is a $1$-design.\\
%(ii)~~~ If this design is a $T(n,3,1,\lambda)$ design, then the intersection array
%of $C$ is $(n,n-1-2\lambda,1;1,n-1-2\lambda,n)$.\\
%(iii)~~ The length $n$ of $C$ is $n = 2^m-1$ for some $m \geq 3$.\\
%(iv)~~~ The code $C \cup C(3)$ is a Hamming code of length $n = 2^m - 1$.
%\end{lemma}

Given a code $C$ with minimum distance $d = 2e+1$, denote by $C^*$ the
extended code, i.e., the code obtained from $C$ by adding an overall parity
checking position. In \cite{bas2} it has been shown when an extension of
an uniformly packed code is again uniformly packed. If this happens the
extended code $C^*$ has the following property.

\begin{lemma}\label{lem:2.7}$\cite{bas2}$
Let $C$ be a uniformly packed code of length $n$ with odd minimum
distance $d$ and let $C_d$ be a $t$-design $T(n, d, t, \lambda)$.
If the extended code $C^*$ is uniformly packed, then the set
$C^*_{d+1}$ is a $(t+1)$-design $T(n+1, d+1, t+1, \lambda)$.
\end{lemma}

Now we give a lemma, which is an strengthening
of a result from \cite{bor3}.

\begin{lemma}\label{lem:2.8}$\cite{bor3}$
Let $C$ be a completely regular linear code of length $n=2^m-1$ with
minimum distance $d = 3$, covering radius $\rho =3$ and
intersection array $(n, b_1, 1; 1, c_2, n)$. Let the dual code
$C^{\perp}$ have nonzero weights $w_i$,~$i=1, 2, 3$. Then the extended
code $C^*$ is completely regular with covering radius $\rho^* = 4$ and
intersection array $(n+1, n, b_1, 1; 1, c_2, n, n+1)$, if
and only if
\EQ
w_1 + w_3 ~=~ 2\,w_2 ~=~ n+1.
\EN
\end{lemma}

\begin{proof}
Let $C$ be given by a parity check matrix $H$. The parity check matrix
$H^*$ of the extended code $C^*$ is obtained from $H$ by adding the
zero column and then the all-one vector. From the condition
$w_1 + w_3 = 2\,w_2 = n+1$, we conclude that the external distance
$s^*$ of $C^*$ equals $s^* = s+1 = 4$. Since $\rho^* = \rho + 1 = 4$
(Lemma \ref{lem:2.1}, (i)), we deduce that $s^* = \rho^*$ and $C^*$ is
uniformly packed (Lemma \ref{lem:2.1},(ii)). If the equalities
$w_1 + w_3 = 2\,w_2 = n+1$ are not satisfied we will have $s^*>4$,
and the code $C^*$ is not
even uniformly packed, and hence it is not completely regular
(Lemma \ref{lem:2.1},(iii)).

To complete the proof, it is enough to compute the intersection array
of $C^*$, which we denote by $(b_0^*,b_1^*,b_2^*,b_3^*;c_1^*,c_2^*,c_3^*,c_4^*)$.

By definition
$$
b^*_0 = n+1,\;\;c^*_4 = n+1.
$$
Since $C^*$ has distance $d^*=4$ we have:
$$
b^*_1 = b_0 = n,\;\;c^*_1 = c_1 = 1.
$$

Since codewords of weight $3$ of $C$
form a design $T(n,3,1,\lambda)$ (Lemma \ref{lem:2.6}) we have that
$b_2^*=b_1 = n-1 - 2\,\lambda$ (Theorem 1 in \cite{bor3}). Now, we show
that $c_2^*=c_2$. Let $\bx\in C(2)$. The number $c_2$ is the number
of cases when the vector $\by$ of weight $3$, at distance one from $\bx$,
is covered by some codewords $\bc\in C$ of weight $4$. Consider $C^*$ and see that  the vector $\bx^*=(0,\,\bx)$ is also in $C^*(2)$. Since
the set of codewords of weight $4$ of $C$ with zero parity check
position is not changed, we conclude that, for this vector $\bx^*$, we have $c_2^*(\bx^*)=c_2$.
Now, for the case when $\bx^*=(1,\,\bx)$
is of weight $2$, we obtain the same value $c_2^*(\bx^*)=c_2$,
for the codewords of $C^*$ of weight $4$ form a $2$-design, i.e., the
number of vectors $\by$, at distance one from $\bx$, covering by some
words from $C^*_4$, does not depend on the choice of $\bx^*$.

Evidently $b_3^*=1=b_2$ and hence $c_3^*=c_3=n$, $c_4^*=n+1$,
finishing the proof.
\end{proof}

\section{Completely regular and completely transitive nested codes}

Recall that $\mathcal{H}_m$ is a binary Hamming code of length $n=2^m-1$.
Assume that $m$ is an even number $m=2u$. Let $q=2^u$, $r=2^u+1$
and $\rr=2^u-1$. We can think of the parity check matrix $H_m$ of
$\mathcal{H}_m$ as the binary representation of
$[\alpha^0, \alpha^1, \ldots, \alpha^{n-1}]$, where $\alpha \in \F_{2^m}$
is a primitive element.

We can present the elements of $\F_{2^m}$ as elements in a quadratic
extension of $\F_{2^u}$. Let $\beta=\alpha^r$  be a primitive
element of $\F_{2^u}$ and let $\F_{2^m}=\F_{2^u}[\alpha]$.

Every element $\gamma\in \F_{2^m}$ can be presented as
$\gamma=\gamma_1+\gamma_2\alpha \in \F_{2^u}[\alpha]$, where
$\gamma_1, \gamma_2 \in \F_{2^u}$. The matrix $H_m$ can also
be written as the binary matrix of size $(2u\times n)$,
where the columns are binary presentations of $[\gamma_i,\gamma_j]$
with $\gamma_i,\gamma_j\in \{0, \beta^1,\ldots, \beta^{q-1}\}$.

\begin{definition}\label{det}
For a given $a=\gamma_1+\gamma_2\alpha$ and $b=\gamma'_1+\gamma'_2\alpha$
from $\F_{2^u}[\alpha]$, define the determinant of $a,b$ in $\F_{2^u}$
as
$$
\Det_u(a,b)= \det \left[ \begin{array}{cc} \gamma_1 &\gamma'_1 \\
\gamma_2 &\gamma'_2 \end {array}\right]= \gamma_1\gamma_2'+\gamma_1'\gamma_2.
$$
\end{definition}

The above definition is the usual definition of determinant.
%So $\det_u(a,b)$
%coincides with the determinant of the matrix giving the homomorphism
%$g:\colon \F_{2^u}^2 \rightarrow \F_{2^u}^2$. Hence,
For a homomorphism
$g:\F_{2^u}^2 \longrightarrow \F_{2^u}^2$ and any two elements $a$ and $b$
from $\F_{2^u}[\alpha]$ we have $\det(g(a),g(b))=\det(g)\det_u(a,b)$,
where $\det(g)$ is the determinant of the matrix defining this homomorphism.
So, if
$$
g=\left[
\begin{array}{cc} g_1 &g'_1 \\
g_2 &g'_2 \end {array}\right],
$$
then $\det(g) = g_1g_2' + g_1'g_2$.

Let $E_m$ be the binary representation of the matrix
$[\alpha^{0r},\alpha^{r},\ldots,\alpha^{(n-1)r}]$. Take the matrix $P_m$
as the vertical join of $H_m$ and $E_m$.

It is well known~\cite{cal} that the code $C^{(u)}$ with parity check
matrix $P_m$ is a cyclic binary completely regular code with covering
radius $\rho=3$, minimum distance $d=3$ and dimension $n-(m+u)$. The
generator polynomial of $C^{(u)}$ is $g(x)=m_\alpha(x)m_{\alpha^{r}}(x)$,
where $m_{\alpha^i}(x)$ means the minimal polynomial associated to $\alpha^i$.

Denote by $\be_i$ the vector with only one nonzero coordinate of value $1$
in position \textit{i}th. Binary vectors $\bv\in \F_2^n$ can be written as
$\bv=\sum_{i\in I_\bv} \be_i$, where $I_\bv=\Supp(\bv)$.
Finite fields $\F_{2^m}$ and $\F_{2^u}[\alpha]$ are isomorphic and so the
elements in $\F_{2^m}$ can be seen as elements in $\F_{2^u}[\alpha]$.
The positions of vectors in $\F_2^n$ can be enumerated by using the
nonzero elements in $\F_{2^m}$  which, in turn, can be seen as elements
in $\F_{2^u}[\alpha]$ by substituting any $\alpha^i\in \F_{2^m}$ with the
corresponding $\alpha^i = \gamma_{i1}+\gamma_{i2}\alpha\in \F_{2^u}[\alpha]$,
where $\gamma_{i1},\gamma_{i2}\in \F_{2^u}$.

For $\bv\in\F_2^n$, vector $H_m\bv^T$ belongs to $\F_2^m$, however we can
consider the representation of it as an element of $\F_{2^u}[\alpha]$,
depending on the context we use the first or the second representation
for $H_m\bv^T$. Clearly we have
$$
H^{}_m \bv^T=H_m (\sum_{i\in I_\bv} \be_i)^T
=\sum_{i\in I_\bv} \alpha^i
=\sum_{i\in I_\bv}(\gamma_{i1}+\gamma_{i2}\alpha).
$$

For any $\bv=\sum_{i\in I_\bv} \be_i\in \F_2^n$, denote
$S(\bv)=\sum_{i\in I_\bv}\gamma_{i1}\gamma_{i2}\in \F_{2^u}$. The next lemma
gives a new description for the code $C^{(u)}$.

\begin{lemma}\label{prop:descrip}
The code $C^{(u)}$ consists of elements $\bv\in \F_2^n$, such that
$H_m \bv^T=0$ and $S(\bv)=0$.
\end{lemma}
\begin{proof}
By definition, a binary vector $\bv$ belongs to $C^{(u)}$, if and only if
$P_m\bv^T=0$, implying $H_m\bv^T=0$ and $E_m\bv^T=0$. Taking
the vector $\bv=\sum_{i\in I_\bv} \be_i$, we are going to prove that conditions
$H_m\bv^T=0$ and $E_m\bv^T=0$ (i.e.,
$\sum_{i\in I_\bv} (\alpha^{i})^r =0$) lead to $S(\bv)=0$.

From the first condition we have $0=H^{}_m\bv^T=\sum_{i\in I_\bv} \gamma_{i1}+\gamma_{i2}\alpha$,
implying that $\sum_{i\in I_\bv} \gamma_{i1}=0$ and $\sum_{i\in I_\bv} \gamma_{i2}=0$.
It also gives $\sum_{i\in I_\bv} \gamma_{i1}^2=0$ and $\sum_{i\in I_\bv} \gamma_{i2}^2=0$.

Now consider the second one:
$$
E^{}_m\bv^T=\sum_{i\in I_\bv} \alpha^{ir} = \sum_{i\in I_\bv} (\gamma_{i1}+\gamma_{i2}\alpha)^r.
$$
Since $r=2^u+1$ and $\gamma_{ik}^{2^u} = \gamma_{ik}$ for $k=1,2$, we obtain
%binomial $(\gamma_{i1}+\gamma_{i2}\alpha)^r$ has only four nonzero
%addends in its development. Therefore,
\begin{eqnarray*}
E^{}_m\bv^T
&=&\sum_{i\in I_\bv} (\gamma_{i1} +\gamma_{i2}\alpha)^{2^u}(\gamma_{i1} +\gamma_{i2}\alpha)\\
&=&\sum_{i\in I_\bv} (\gamma_{i1} +\gamma_{i2}\alpha^{2^u})(\gamma_{i1} +\gamma_{i2}\alpha)\\
%&=& \sum_{i\in I_\bv} (\gamma_{i1}^r +\gamma_{i2}^r\alpha^r+ \gamma_{i1}^{r-1}
%\gamma_{i2}\alpha+\gamma_{i1}\gamma_{i2}^{r-1}\alpha^{r-1} )\\
&=&\sum_{i\in I_\bv} (\gamma_{i1}^2 +\gamma_{i2}^2\beta+ \gamma_{i1}\gamma_{i2}(\alpha+\alpha^{r-1}))
\end{eqnarray*}
(recall that $\beta=\alpha^r$) and, since $H_m\bv^T=0$ and $\alpha+\alpha^{r-1}\not=0$, we finally obtain $E^{}_m\bv^T= 0$
if and only if
$$S(\bv)=\sum_{i\in I_\bv} \gamma_{i1}\gamma_{i2}=0.
$$
 \end{proof}

The code $C^{(u)}$ is a binary $[n=2^m-1,k=n-m-u]$ code and it is a subcode of the
$[2^m-1,n-m]$ Hamming code $\mathcal{H}_m$. Now we show, that
$C^{(u)}$ is not only completely regular ~\cite{cal}, but
also completely transitive.

An isomorphism $\Phi: \F_{2^u}^2\longrightarrow \F_{2^u}^2$ is given by
a ($2\times 2$)-matrix over $\F_{2^u}$,
$$
\Phi = \left[
\begin{array}{cc}
a~&~a'\\
b~&~b'
\end{array}
\right],
$$
with nonzero determinant $\det(\Phi) = ab'+a'b \neq 0$, such that
$$\Phi (\gamma_{i1}, \gamma_{i2})^T = (a\gamma_{i1}+a'\gamma_{i2},\;b\gamma_{i1}+b'\gamma_{i2})^T =
(\gamma_{j1}, \gamma_{j2})^T.$$

The above isomorphism $\Phi$ induces a permutations  of columns, denoted by $\varphi$,
where the column $\alpha^i = \gamma_{i1} + \gamma_{i2} \alpha$ is moved under the action
of $\varphi$  to the column $\alpha^j = \gamma_{j1} + \gamma_{j2} \alpha$, i.e., $\varphi((\gamma_{i1}, \gamma_{i2})^T)=(\gamma_{j1}, \gamma_{j2})^T$.

The above presentation implies that the general linear group $\operatorname{GL}_2(2^u)$
stabilizes $C^{(u)}$.

\begin{proposition}\label{aut}
The automorphism group of $C^{(u)}$  contains the linear group
$\operatorname{GL}_2(2^u)$, so
$\langle GL_2(2^u)\rangle \subseteq \operatorname{Aut}(C^{(u)})$.
\end{proposition}

\begin{proof}
Let $\Phi \in \operatorname{GL}_2(2^u)$ and, as we said before, consider the
associated permutation $\varphi\in \operatorname{S}_n$. We want to see that $\varphi\in \operatorname{Aut}(C^{(u)})$.

Let $\bv = \sum_{i\in I_\bv} \be_i \in C^{(u)}$, hence  from Lemma \ref{prop:descrip},
$H_m\bv^T=0$ and $S(\bv)=0$. Thus,
$$
\sum_{i\in I_\bv} \gamma_{i1}=0,\;\;\sum_{i\in I_\bv} \gamma_{i2}=0\;\;\mbox{and}\;\;
\sum_{i\in I_\bv} \gamma_{i1}\gamma_{i2}=0,
$$
where for $i\in \{0,\ldots,n-2\}$ we have
$\alpha^i = \gamma_{i1}+\gamma_{i2}\alpha \in \F_{2^u}[\alpha]$. Also we have
$\sum_{i\in I_\bv} \gamma_{i1}^2=0$ and $\sum_{i\in I_\bv} \gamma_{i2}^2=0$.

Now we have to prove that $H_m(\varphi(\bv))^T = S(\varphi(\bv))=0$. We obtain
$$
H_m(\varphi(\bv))^T = \sum_{i\in I_\bv} \Phi(\gamma_{i1},\gamma_{i2})^T =
\sum_{i\in I_\bv} (a\gamma_{i1}+a'\gamma_{i2}) +(b\gamma_{i1}+b'\gamma_{i2})\alpha=0.
$$
and
\begin{eqnarray*}
S(\varphi(\bv))&=&\sum_{i\in I_\bv} (a\gamma_{i1}+b\gamma_{i2})(a'\gamma_{i1}+b'\gamma_{i2})\\
&=&\sum_{i\in I_\bv} aa'\gamma_{i1}^2+bb'\gamma_{i2}^2+ (ab'+a'b)\gamma_{i1}\gamma_{i2}\\
&=&\det(\Phi)\sum_{i\in I_\bv} \gamma_{i1}\gamma_{i2} = 0.
\end{eqnarray*}
\end{proof}

\begin{proposition}
The automorphism group of $C^{(u)}$ gives four orbits on the cosets of
$C^{(u)}$ in $\F_2^n$ and so $C^{(u)}$ is a completely transitive code.
\end{proposition}

\begin{proof}
Denote the syndrome of any vector $\bv\in \F_2^n$ as $[h,e]$,
where $e=S(\bv)\in \F_{2^u}$ and $h=H^{}_m\bv^T$. Since $C^{(u)}$ has
covering radius $\rho=3$ we have four different classes of
cosets of $C^{(u)}$ depending on their weight. The coset of weight $0$
coincides with $C^{(u)}$, so its vectors  have syndrome $[0,0]$. The
cosets of weight $1$ are those with syndrome $[h,e]$, where
$h=H^{}_m\bv^T$ for
some vector $\bv$ of weight one such that $e=S(\bv)$, hence a total
of $2^m-1=r\rr$ cosets. Since $\operatorname{GL}_2(\F_{2^u})$ is
transitive over the set $\{\be_i:1\leq i \leq n-1\}$ the orbit of a vector
in a coset of weight one covers all cosets of weight one.

The cosets of weight $3$ are those with syndrome $[h,e]$, where
$h=H^{}_m\bv^T=0$ and $e = S(\bv)\not=0$. As we saw in the preamble
of Lemma~\ref{prop:descrip}, $e\in  \F_{2^u}$, so it has $\rr$
possible values and there are a total of $\rr$ cosets of weight $3$.
Like for the above case when the cosets are of weight $1$, the orbit
of a vector in a coset of weight $3$ contains all cosets of weight $3$.
Indeed, from Proposition~\ref{aut} there exists an automorphism
with the appropriate determinant which takes $e$ to any other possible $e'$.

The cosets of weight $2$ are those with syndrome $[h,e]$, where
$h=H^{}_m\bv^T\not= 0$ and
$e\in \F_{2^u}\backslash \{z\}$, where $z=S(\bv)$. Hence, a
total of $(2^m-1)\rr = r\rr^2$ cosets.  The representatives
in all cosets of weight $2$ are vectors of weight two, which can be
seen as pairs $a,b$, where $a=(a_1,a_2)\in \F_{2^u}^2$, $b=(b_1,b_2)\in \F_{2^u}^2$,
such that $h=a+b$ and $S(h)\not=S(a)+S(b)$.
We have
$$
S(h) = S(a_1+b_1, a_2+b_2)) = (a_1+b_1)(a_2+b_2) = S(a)+S(b)+det_u(a,b),
$$
so the condition $S(h)\not=S(a)+S(b)$ is equivalent to the condition $\det_u(a,b)\not=0$.
Therefore, the cosets of weight two are those with representative
pairs $a,b$ with $\det_u(a,b)\not=0$. Given two pairs $a,b$ and $c,d$,
with $\det_u(a,b)\not=0$ and $\det_u(c,d)\not=0$ from Proposition~\ref{aut}
we always can find an
isomorphism of $\F_{2^u}^2$ taking $a,b$ to $c,d$ and so, an automorphism
of $C^{(u)}$ sending the coset with representative pair $a,b$ to the
coset with representative pair $c,d$.
\end{proof}

As we know, the number of cosets $C^{(u)}+\bv$, of weight three, is $\rr$.
Indeed, their syndromes $S(\bv)$ are the nonzero elements of $\F_{2^u}$.
For $i\in \{0,\ldots,u\}$, taking $u-i$ cosets $C^{(u)}+\bv_1, \ldots, C^{(u)}+\bv_{u-i}$
with independent syndromes
$S(\bv_1), \ldots, S(\bv_{u-i})$ (independent, means that they are
independent binary vectors in $\F_2^u$) we can generate a linear binary
code $C^{(i)}=\langle C^{(u)}, \bv_1, \ldots \bv_{u-i}\rangle$.

The dimension of code $C^{(i)}$ is $\dim(C^{(i)})= u-i+\dim(C^{(u)})$,
where $\dim(C^{(u)})=n-m-u$. Note that the maximum number of independent
syndromes we can take is $u$, so the biggest code we can obtain is of
dimension $u+\dim(C^{(u)})=n-m$, which is the Hamming code $C^{(0)}=\mathcal{H}_m$.
All the constructed codes contains  $C^{(u)}$ and, at the same time,
they are contained in the Hamming code $C^{(0)}$.

The number of codes $C^{(u-i)}$ equals the number of subspaces of dimension
$i$ we can take in $\F_2^u$, so the Gaussian binomial coefficient
$$
|\{C^{(u-i)}\}|=\binom{u}{i}_2=\frac{(2^u-1)(2^u-2)\cdots(2^u-2^{i-1})}{(2^i-1)(2^i-2)\cdots(2^i-2^{i-1})}.
$$
Taking all the possibilities, we are able to construct several nested families
of codes between $C^{(u)}$ and $C^{(0)}=\mathcal{H}_m$. In fact, it is easy to
compute that there are
$$
\prod_{i=0}^{u-1} (2^{u-i}-1)
$$
different families.

All these codes $C^{(i)}$ are completely regular as we show later in
Theorem~\ref{CRIA}. We have seen that $C^{(u)}$
and $C^{(0)}$ are completely transitive and, in addition we show that also
$C^{(1)}$ is also a completely transitive code.

\begin{proposition}
The automorphism group of $C^{(1)}$ induces $4$ orbits on the cosets of $C^{(1)}$
in $\F_2^n$ and so $C^{(1)}$ is a completely transitive code.
\end{proposition}

\begin{proof}
Let $\F_2^u=A_{u-1}\oplus A_1$ be the decomposition of the binary linear space
$\F_2^u$ as a direct sum of subspaces and let $S_1(\bv)$ be the projection of
$S(\bv)$ over $A_1$. By definition of $C^{(1)}$ the elements $\bv\in C^{(1)}$ are
those such that $H_m\bv^T=0$ and $S(\bv)$
belongs to a subspace $A_{u-1} \subset \F_2^u$ of dimension $u-1$ over $\F_2$.
Therefore, the elements of $C^{(1)}$ can be characterized
by the syndrome $h=H^{}_m\bv^T=0$ and $e=S_1(\bv)=0$. Following the same
argumentation and computations as in Proposition~\ref{aut}, we easily obtain
that $\operatorname{SL}_2(2^u)\subset \Aut(C_1)$, where $\operatorname{SL}_2(2^u)$
is the special linear group of automorphisms, so  the normal subgroup of the
general linear group $\operatorname{GL}_2(2^u)$, consisting of those matrices
$\Phi$ with determinant $\det(\Phi)=1$.

The cosets of $C^{(1)}$ of weight $1$ are those with syndrome $[h,e]$, where
$h=H^{}_m\bu^T\in \F_{2^u}[\alpha]\backslash\{0\}$ for some vector
$\bu\not=0$  of weight one such that $e=S_1(\bu)$, hence a total of
$2^m-1=r\rr$ cosets. Since $\operatorname{SL}_2(2^u)$ is transitive
over $\{\be_i:1\leq i \leq n-1\}$ the orbit of a vector in a coset of
weight one contains all cosets of weight one.

There is only one coset of weight $3$, say $C^{(1)}+ \bu$, where
$H^{}_m\bu^T=0$ and $S(\bu)\in A_1$. Hence, there is nothing to
prove, automorphisms of $C^{(1)}$ act transitively over this unique coset.

The cosets of weight $2$ are $C^{(1)}+\bu+\bv$, where $C^{(1)}+\bu$ is
the coset of weight three and $\bv$ is of weight
one. The syndrome of these cosets is  $[h,e]$, where
$h=H^{}_m\bv^T\not=0$ and
$e=S_1(\bu+\bv)=S_1(\bv)$. We have a total of $2^m-1=r\rr$ cosets of weight two.
Like for the cosets of weight one, since $\operatorname{SL}_2(2^u)$ is
transitive over $\{\be_i:1\leq i \leq n-1\}$ the orbit of a vector in a
coset of weight two cover all cosets of weight two.
\end{proof}

As a generalization of the previous proposition we can state, as a conjecture,
the following proposition which needs the exact computation of the automorphism
group of any $C^{(i)}$ to be solved.

\begin{conjecture}
Code $C^{(i)}$ is completely transitive if and only if $i=0, i=1, i=u$ or $2^i \leq u+1$,
for $i\in \{2,\ldots,u-1\}$.
\end{conjecture}

Note that for $m=6$ (so $u=3$), all codes $C^{(i)}$ in the chain are completely transitive. Thus, the conjecture is true for this case.

Finally, we can prove that all codes $C^{(i)}$ are completely regular.

\begin{lemma}\label{c3}
Let $C_3^{(i)}$ be the set of all codewords in $C^{(i)}$ of weight three.
Then $C_3^{(i)}$ is a $T(n,3,1,\lambda_i)$ design, where $\lambda_i=2^{m-i-1}-1$.
\end{lemma}

\begin{proof}
From the construction of codes $C^{(i)}$ we know that the codewords
$\bv$ of weight three are those such that $H_m\bv^T=0$ and $S(\bv)$
belongs to a fixed subspace $A_{u-i} \subset \F_2^u$ of dimension $u-i$
over $\F_2$. Hence, taking a fixed nonzero element
$\gamma=\gamma_1+\gamma_2\alpha \in \F_{2^m}$ every codeword of weight
three covering this element $\gamma$ is defined giving
$\gamma'=\gamma'_1+\gamma'_2\alpha \in \F_{2^m}$ such that
$\det_u(\gamma,\gamma')\in A_{u-i}$. Indeed, if $\bv\in \F_2^n$
is of weight three let
$$
\{\gamma=\gamma_1+\gamma_2\alpha,\; \gamma'=\gamma'_1+
\gamma'_2\alpha,\;\gamma''=\gamma''_1+\gamma''_2\alpha \}
$$
be its support. Then, since $H_m\bv^T=0$ we have $\gamma''=\gamma+\gamma'$ and so
$$
S(\bv)=\gamma_1\gamma_2+ \gamma'_1\gamma'_2+\gamma''_1\gamma''_2 =\Det_u(\gamma,\gamma').
$$

Now we want to count how many codewords of weight three cover a fixed
nonzero element $\gamma\in \F_{2^m}$. We begin by counting
how many $\gamma'\in \F_{2^m}$ gives $\Det_u(\gamma,\gamma') \in A_{u-i}$. Recall that $\Det_u(\gamma,\gamma')$ is an element of $\F_{2^u}$, considered as a binary vector.
For any nonzero element $\gamma=\gamma_1+\gamma_2\alpha \in \F_{2^m}$ there
are $2^u-2$ nonzero values $\gamma'=\beta^i\gamma\not=\gamma$, where
$i\in\{0,\ldots,2^u-2\}$ such that $\Det_u(\gamma, \gamma')=0$ and there
are $2^u$ values $\gamma'$ giving $\Det_u(\gamma, \gamma')=\beta^j\in \F_{2^u}$, for a fixed  $j\in\{0,\ldots,2^u-2\}$.
There are $2^{u-i}-1$ nonzero vectors in $A_{u-i}$. Hence, summing up,
we conclude that there are $2^u-2 + 2^u(2^{u-i}-1)=2^{m-i}-2$ values
$\gamma'$, such that $\det(\gamma, \gamma') \in A_{u-i}$.
However, the codeword with support $\{\gamma, \gamma', \gamma''=\gamma+\gamma'\}$
is counted twice, once as $\gamma'$ and again as $\gamma''$. Hence, finally,
the number $\lambda_i$ of codewords of weight three covering $\gamma$ is $2^{m-i-1}-1$.
\end{proof}

\begin{theo}\label{CRIA}
For $i\in \{0,\ldots,u\}$, the code $C^{(i)}$ is completely regular with
intersection array
$(2^m-1,2^{m}-2^{m-i},1;1,2^{m-i},2^m-1).$
\end{theo}

\begin{proof}
For each $i\in \{0,\ldots,u\}$, we have that $C^{(i)}$ is completely regular
if the parameters $(b_0, b_1, b_2; c_1,c_2,c_3)$ of the intersection array
are computable. Since the minimum distance in $C^{(i)}$ and in
$C^{(i)}(\rho)=C^{(i)}(3)$ is 3, it is obvious that $b_0=c_3=n=2^m-1$ and
$b_2=c_1=1$. Let $\bx\in C(1)$, we count the number of neighbors of $\bx$ in
$C^{(i)}(1)$. Without loss of generality, we assume that $\bx$ has
weight one. Therefore, $\bx$ has $n-1$ neighbors of weight two. By Lemma \ref{c3}, twice
$\lambda_i$ of these neighbors are covered by minimum weight
codewords of $C^{(i)}$. As the result does not depend on the choice of
$x$, we conclude that $a_1=2\lambda_i$. Therefore, $b_1=n-c_1-a_1=n-1-\lambda_i=2^{m}-2^{m-i}$.
A similar argument shows that any vector in $C^{(i)}(2)$ has a fixed
number of neighbors in $C^{(i)}(2)$. Therefore, $c_2$ is also
calculable. Applying Lemma \ref{lem:2.2}, we have that $\mu_1 b_1 =
\mu_2 c_2$. Since:
$$
\mu_0=1;\;\mu_1=n;\;\mu_3=2^i-1\;\mbox{ and
}\;\mu_0+\mu_1+\mu_2+\mu_3=2^i(n+1);
$$
we deduce $\mu_2=(2^i-1)n$ and $c_2= \mu_1 b_1 / \mu_2=b_1/(2^i-1)=2^{m-i}$.
\end{proof}

\begin{coro}\label{extended}
For $i\in \{0,\ldots,u\}$, the extended code  $C^{(i)*}$ is completely regular
with intersection array $(2^m,2^m-1,2^{m}-2^{m-i},1;1,2^{m-i},2^m-1,2^m).$
\end{coro}

\begin{proof}
By Theorem \ref{CRIA} any code $C^{(i)}$ is completely regular. In particular,
this means (Lemma \ref{lem:2.1}), that for any such code $C^{(i)}$ the
external distance $s(C^{(i)})$ equals the covering radius $\rho(C^{(i)})$,
i.e. $s(C^{(i)}) = \rho(C^{(i)})$. Since $\rho(C^{(i)}) = 3$, we conclude that
$s(C^{(i)}) = 3$ for any $i\in \{0,\ldots,u\}$. As it was shown in \cite{cal},
the dual code of $C^{(u)}$ has the following values in the weight spectrum:
$$
%\frac{1}{2}~(n+1), \;\;\frac{1}{2}~(n+1) \pm \frac{1}{2}~\sqrt{n+1},
2^{m-1}, \;\;2^{m-1} \pm 2^{u-1}.
$$
%where $n=2^{2u}-1$.
But any code $C^{(i)}$ contains the code $C^{(u)}$ as
a subcode, implying that the dual $(C^{(i)})^\perp$ is contained in
$(C^{(u)})^\perp$. This, in turn, implies that any such code $(C^{(i)})^\perp$
has the same weight spectrum as the code $(C^{(u)})^\perp$.
Now the result follows from Lemma \ref{lem:2.8}.
\end{proof}

The next theorem shows that the extended codes $C^{(i)*}$ are not only completely regular, but completely transitive.
\begin{theo}\label{dt}
For $i\in \{0,\ldots,u\}$, the automorphism group of the extended code  $C^{(i)*}$ is $\Aut(C^{(i)*})=\Aut(C^{(i)})\ltimes \F_2^m$. Code $C^{(i)*}$ is completely transitive when $C^{(i)}$ is completely transitive.
\end{theo}
\begin{proof}
Let $\F_2^u=A_{u-i}\oplus A_i$ be the decomposition of the binary linear space
$\F_2^u$ as a direct sum of subspaces of dimension $u-i$ and $i$, respectively. Let $S_i(\bv)$ be the projection of
$S(\bv)$ over $A_i$.
By definition, the elements of $\bv\in C^{(i)}$ can be characterized
by the syndrome $H^{}_m\bv^T=0$ and $S_i(\bv)=0$.

Code $C^{(i)*}$ is the extension of $C^{(i)}$ by an overall parity check coordinate, which we assume is the $0$th coordinate.
Codewords in $C^{(i)*}$ have $n=2^m$ components and we can associate, at random and once for all, the coordinate $i$th with a vector $\bw_i\in \F_2^m$.
Any vector $\bw\in\F_2^m$ define a permutation $\pi_\bw\colon \{1,\ldots,n\}\longrightarrow \{1,\ldots,n\}$ such that $\pi_\bw(i)=j$, where $\bw_j=\bw+\bw_i$. Let $T=\{\pi_\bw : \bw\in \F_2^m\}$ the set of all these permutations and note that $T$ has a group structure isomorphic to the additive structure $\F_2^m$. For each $\bw\in \F_2^m$, set $\bw=\gamma_1+\gamma_2\alpha \in \F_{2^u}[\alpha]$.

As all codewords in $C^{(i)*}$ have even weight it is clear that $T$ is a subgroup of $\Aut(C^{(i)*})$. Indeed, let $\ba=(a_0,\ldots,a_n)\in C^{(i)*}$, this means that $\ba$ has an even number of nonzero components ($\sum_{i=0}^n a_i=0$); $H_m\ba^T=\sum_{\bw_i\in \F_2^m} a_i \bw_i =\bo$ and $S(\ba)\in A_{u-i}$. Now take $\pi_\bw(\ba)=\ba'=(a_0',a_1',\ldots,a_n')$, where $a'_j=a_i$, such that $\bw_j=\bw+\bw_i$  and compute:
\begin{equation*}
\begin{split}
\sum_{j=0}^n a'_j=&\sum_{i=0}^n a_i=0;\\
H_m\ba'^T=&\sum_{\bw_j\in \F_2^m} a'_j \bw_j =\sum_{\bw_i\in \F_2^m} a_i (\bw+\bw_i)= \bw\sum_{i=0}^n a_i=\bo\\
S(\ba')= &\sum_j a'_j(\gamma_{j1}\gamma_{j2}) =\sum_i a_i (\gamma_1+\gamma_{i1})(\gamma_2+\gamma_{i2})= \\ &\sum_i a_i \gamma_{i1}\gamma_{i2} + \gamma_1 \sum_i a_i \gamma_{i1} +\gamma_2 \sum_i a_i \gamma_{i2} +\gamma_1\gamma_2\sum_i a_i = S(\ba) \in A_{u-i}.
\end{split}
\end{equation*}
Hence, $\pi_\bw(\ba)\in C^{(i)*}$.

 Furthermore, $T$ is a normal subgroup in $\Aut(C^{(i)*})$. Indeed, for any $\phi\in \Aut(C^{(i)*})$ we have that $\phi \pi_\bw\phi^{-1}$ is again a translation $\pi_\bz$, where $\bz=\phi(\bw)$.
 For any $\phi\in\Aut(C^{(i)*})$, it is clear that we can find $\phi'\in\Aut(C^{(i)})$ fixing the extended coordinate and a vector $\by\in \F_2^m$, such that $\phi=\phi'\pi_{\by}$.
 Therefore, we have $\Aut(C^{(i)*})/T \cong \Aut(C)$ and so $\Aut(C^{(i)*})$ is the semidirect product of $\F_2^m$ and $\Aut(C^{(i)})$ (obviously, we can identify $T$ with $\F_2^m$).
 The first statement is proven.

Let us assume that $C^{(i)}$ is completely transitive. To prove that $C^{(i)*}$ is completely
transitive we show that all cosets of $C^{(i)*}$ in $\F_2^{2^m}$ with the same minimum weight are in the same orbit by the action of $\Aut(C^{(i)*})$.

The number of cosets of $C^{(i)*}$ is twice the cosets of $C^{(i)}$. If $C^{(i)}+\bv$ is a coset of $C^{(i)}$, where $\bv$ is a representative vector of minimum weight then $C^{(i)*}+(0|\bv)$ and $C^{(i)*}+(1|\bv)$ are cosets of $C^{(i)*}$. The cosets of $C^{(i)*}$ of weight $4$ are of the form $C^{(i)*}+(1|\bv)$, where $C^{(i)}+\bv$ is a coset of weight $3$ of $C^{(i)}$. Since the cosets of weight $3$ of $C^{(i)}$ are in the same $\Aut(C^{(i)})$-orbit and $\Aut(C^{(i)})\subset \Aut(C^{(i)*})$, it follows that all the cosets of weight $4$ of $C^{(i)*}$ are in the same $\Aut(C^{(i)*})$-orbit.

Now consider the cosets of $C^{(i)*}$ of weight $r\in\{1,2,3\}$.  They are of the form $C^{(i)*}+(0|\bv)$, where $C^{(i)}+\bv$ is a coset of weight $r$ of $C^{(i)}$ and of the form $C^{(i)*}+(1|\bv)$, where $C^{(i)}+\bv$ is a coset of weight $r-1$ of $C^{(i)}$.
Cosets of the same minimum weight in $C^{(i)}$ can be moved among them by $\Aut(C^{(i)})$ and so, we need only to show that there exists an automorphism in $\Aut(C^{(i)*})$ moving $C^{(i)*}+(0|\bv)$ to $C^{(i)*}+(1|\bv')$, where $\bv,\bv'$ are at distance $r$ and $r-1$ from $C^{(i)}$, respectively.
Without loss of generality, we further assume that $\Supp(\bv')\subset \Supp(\bv)$ and so, $\Supp(\bv)=\Supp(\bv')\cup \{j\}$, for some index $j\in \{1,\ldots,2^m\}$.
The automorphism $\pi_{\bw_j}$  moves $C^{(i)*}+(0|\bv)$ to $C^{(i)*}+(1|\bv'')$, where $\Supp(\bv'')=\{k\,:\, \bw_k=\bw_j+\bw_s; s\in \Supp(\bv')\}$ and, finally, by using an automorphism from $\Aut(C^{(i)})$ we can move from $C^{(i)*}+(1|\bv'')$ to $C^{(i)*}+(1|\bv')$.
\end{proof}

\section{Nested antipodal distance-regular graphs and distance-transitive graphs of diameter $3$ and $4$}

Denote by $\Gamma^{(i)}$ (respectively, $\Gamma^{(i)*}$) the coset
graph, obtained from the code $C^{(i)}$ (respectively $C^{(i)*}$)
by Lemma \ref{lem:2.5}.

Since all cosets of weight $3$ (respectively, of weight $4$) of
the Hamming code $\mathcal{H}_m$ (respectively, of the extended
Hamming code $\mathcal{H}_m^*$) belong to this code, we conclude
that all graphs $\Gamma^{(i)}$ (respectively, $\Gamma^{(i)*}$) are
antipodal. This means that for $i>0$ all graphs $\Gamma^{(i)}$ and
$\Gamma^{(i)*}$ are imprimitive.

We need the following statement from \cite{god}.
\begin{lemma}\label{lem:cover}
Let $\Gamma$ be an antipodal distance-regular graph of diameter three.
Then $\Gamma$ is a $r$-fold covering graph of $K_n$, for some $r$ and $n$ and recall that $c_2$ is the number of common neighbors of two vertices in $\Gamma$ at distance two. Then the intersection array of $\Gamma$ is $(n-1,(r-1)c_2,1;1,c_2,n-1)$.
\end{lemma}

As a direct result of Lemma~\ref{lem:cover} and
Theorem \ref{CRIA} we obtain the following
new distance-regular and distance-transitive coset graphs.

\begin{theo}\label{grafs1}
For any even $m=2u$, ~$m\geq 4$
there exist a family of embedded antipodal distance-regular
coset graphs $\Gamma^{(i)}$ with $2^{2u+i}$ vertices and diameter $3$,
for $i=1, \ldots, u$. Graph $\Gamma^{(0)}$ has diameter $1$, i.e., it is a complete
graph $K_n$,\,$n=2^m-1$. Specifically:
\begin{itemize}
\item $\Gamma^{(i)}$, \,$i=1, \ldots, u$ has intersection array
$$
(2^m-1, 2^{m}-2^{m-i}, 1;1, 2^{m-i}, 2^m-1).
$$
\item $\Gamma^{(i)}$ is a subgraph of $\Gamma^{(i+1)}$ for all
$i=0,1, \ldots, u-1$.
\item $\Gamma^{(i)}$ covers $\Gamma^{(j)}$, where $j\in\{0,1, \ldots, i-1\}$ with
parameters $(2^m-1,2^{i-j},2^{2u-i+j})$,
\item The graphs $\Gamma^{(i)}$  are distance-transitive
for $i\in \{0,1,u\}$ when $m\geq 8$ and for $i\in \{0,1,2,3\}$ when $m=6$.
%\item All graphs $\Gamma^{(i)}$ with $i\geq 1$ are imprimitive.
\end{itemize}
\end{theo}

As for the codes that give rise to this graphs, we conjecture that the graphs $\Gamma^{(i)}$ are
distance-transitive for $i\in \{2,\ldots,u-1\}$ and $2^i \leq
u+1$.

Finally, from  Lemma \ref{lem:2.5}, Theorem~\ref{dt} and Corollary~\ref{extended}, we can establish the following results for the coset graphs coming from the extended codes $C^{(i)*}$.

\begin{theo}\label{grafs2}
For any even $m=2u$, ~$m\geq 4$, $n=2^m-1$ and any $i=0,1, \ldots, u$
there exist a family of embedded antipodal distance-regular
coset graphs $\Gamma^{(i)*}$ with $2^{m+i+1}$
vertices and diameter $4$.  Specifically:
\begin{itemize}
\item $\Gamma^{(i)*}$ has intersection array
$$
(2^m+1,2^m, 2^{m}-2^{m-i}, 1;1, 2^{m-i},2^m, 2^m+1).
$$
\item $\Gamma^{(i)*}$ is a subgraph of $\Gamma^{(i+1)*}$ for all
$i=0,1, \ldots, u-1$.
\item $\Gamma^{(i)*}$ covers $\Gamma^{(j)*}$, where $j=0,1, \ldots, i-1$ with
the size of the fibre $r_{i,j}=2^{i-j}$.
\item The graphs $\Gamma^{(i)*}$  are distance-transitive
for $i=0,1,u$ when $m\geq 8$ and $i=0,1,2,3$ when $m=6$.
%\item All graphs $\Gamma^{(i)*}$  are imprimitive.
\end{itemize}
\end{theo}

We also conjecture that the graphs $\Gamma^{(i)*}$  are
distance-transitive for $i\in \{2,\ldots,u-1\}$ and $2^i \leq
u+1$.

The first graphs $\Gamma^{(1)}$ and $\Gamma^{(1)*}$ are
well known distance-transitive graphs (see \cite{bor3,bor4} and
references there).

Graphs $\Gamma^{(u)}$ and $\Gamma^{(u)*}$ are also known. The
corresponding codes $C^{(u)}$ and $C^{(u)*}$ have been constructed by
Kasami \cite{kas} and have been presented in a very symmetric form
by Calderbank and Goethals \cite{cal}. They proved that these codes
form an association scheme \cite{del1}, which immediately implies the
existence of the corresponding distance-regular graphs $\Gamma^{(u)}$
and $\Gamma^{(u)*}$ (Ch. 11 in \cite{bro2}).

All graphs $\Gamma^{(i)}$ for $i=0,1, \ldots, u$ have been
constructed by Godsil and Hensel using the Quotient Construction \cite{god}.
But it was not mentioned in all references above that some of these graphs are
completely transitive. Besides, except for the graphs $\Gamma^{(u)}$, it was
not stated that these graphs can be constructed as coset graphs.

The graphs $\Gamma^{(i)*}$ for $i=2, \ldots, u-1$ seems to be new; we
could not find graphs with these parameters
in the above mentioned literature.

\end{document}